\documentclass[reqno,oneside,11pt]{amsart}
\usepackage[T1]{fontenc}
\usepackage[utf8]{inputenc}
\usepackage{lmodern}
\usepackage[french,english]{babel}
\usepackage{geometry} 

\usepackage{amsmath,amssymb,amsfonts,amsthm,mathtools}
\usepackage{mathrsfs}
\usepackage{graphicx}
\usepackage{caption}
\usepackage{subcaption} 
\usepackage{booktabs}
\usepackage{array}
\usepackage{paralist}
\usepackage{fancyhdr}
\usepackage{emptypage}
\fancypagestyle{plain}{
	\fancyhf{}
	\fancyfoot[RO,RE]{\thepage}
	
	}

\usepackage{comment} 
\usepackage{diagbox}
\usepackage{xcolor}
\usepackage{epigraph}

\usepackage[noadjust]{cite}
\usepackage[
pagebackref=true,
colorlinks=true,
urlcolor=purple,
linkcolor=purple!87!black,
citecolor=green!60!black,
pdfborder={0 0 0}
]{hyperref}
\renewcommand*{\backref}[1]{}
\renewcommand*{\backrefalt}[4]{[{\tiny%
		\ifcase #1 Not cited.%
		\or Cited on page~#2.%
		\else Cited on pages #2.%
		\fi%
	}]}
\usepackage{cleveref}

\usepackage{bookmark}
\usepackage{dsfont}

\newcommand{\Z}{\mathbb{Z}}

\newcommand{\R}{\mathbb{R}}
\renewcommand{\P}{\mathbb{P}}

\newcommand{\supp}[1]{\mathrm{supp}(#1)}

\makeatletter
\newcommand{\neutralize}[1]{\expandafter\let\csname c@#1\endcsname\count@}
\makeatother

\newcommand{\thistheoremname}{}

\newtheorem*{genericthm*}{\thistheoremname}
\newenvironment{namedthm*}[1]
{\renewcommand{\thistheoremname}{#1}%
	\begin{genericthm*}}
	{\end{genericthm*}}
\theoremstyle{plain}
\newtheorem{thm}{Theorem}[section] 
\newtheorem{prop}[thm]{Proposition}
\newtheorem{lem}[thm]{Lemma}

\newtheorem{question}[thm]{Question}
\theoremstyle{definition}
\newtheorem{rem}[thm]{Remark}
\newtheorem{defn}[thm]{Definition} 
\newtheorem{exmp}[thm]{Example} 
\allowdisplaybreaks
\author{Eduardo Silva} 
\email[Eduardo Silva]{eduardo.silva@uni-muenster.de, edosilvamuller@gmail.com}
\address{University of Münster, Einsteinstrasse 62, Münster 48149, Germany}
\urladdr{\url{https://edoasd.github.io/eduardo_silva_math/}}
\title[Continuity of asymptotic entropy on free solvable groups]{Continuity of asymptotic entropy on free solvable groups}
\date{\today}

\newcommand{\pesc}[1]{p_{\mathrm{esc}}(#1)}

\begin{document}
	\begin{abstract} 
		We prove the continuity of asymptotic entropy, as a function of the step distribution, among non-degenerate probability measures with finite Shannon entropy on the free solvable group $S_{d,m}$ of rank $d\ge 3$ and derived length $m\ge 2$.
	\end{abstract}
	\maketitle
\section{Introduction}\label{section: introduction}

One of the earliest results relating geometric properties of non-abelian infinite groups to random walks is due to H.~Kesten \cite{Kesten1959}, who showed that a group is non-amenable if and only if the return probability to the origin at time $n$ of a symmetric random walk on the group decays exponentially as $n\to\infty$. Later, J. Rosenblatt \cite[Theorem 1.10]{Rosenblatt1981} and independently V. Kaimanovich and A. Vershik \cite[Theorem 4.4]{KaimanovcihVershik1983} proved that a group is amenable if and only if there exists a probability measure $\mu$ on the group with the Liouville property, meaning that every bounded $\mu$-harmonic function is constant. Recall that a function $f:G\to\mathbb{R}$ is called \emph{$\mu$-harmonic} if $ f(g)=\sum_{h\in G} f(gh)\mu(h)$ for all $g\in G$. V. Kaimanovich and A. Vershik also exhibited the first examples of amenable groups that admit symmetric finitely supported probability measures that do not possess the Liouville property \cite[Proposition 6.4]{KaimanovcihVershik1983}. Their examples are the wreath products $\mathbb{Z}/2\mathbb{Z}\wr \mathbb{Z}^d$ for $d\ge 3$, which are not only amenable but even $2$-step solvable.

Within the class of amenable groups, solvable groups provide a particularly rich source of examples exhibiting diverse behavior for random walks. Most notably, J. Brieussel and T. Zheng showed that any sufficiently regular function can be realized as the asymptotic drift and asymptotic entropy (see Definition \ref{defn: asymptotic entropy} below) of a random walk on a $4$-step solvable group \cite[Theorem 1.1]{BrieusselZheng2021}. Their result generalizes earlier results in this direction by \cite[Theorem 1]{Erschlerasymptotics}, \cite[Theorems 1 \& 2]{Erschlerdrift}, \cite[Theorem 1.1]{KotowskiVirag2015}, \cite[Theorems 1 \& 3]{AmirVirag2017}, and \cite[Theorem 1]{Amir2017}, which are also based on solvable groups.

Among solvable groups, a natural class to study is given by free solvable groups, which may be viewed as the solvable groups of a given rank and derived length with no relations other than those imposed by solvability of that derived length.
\begin{defn}\label{def: free solvable}
	Denote by $F_d$ the free group of rank $d\ge 2$. The \emph{free solvable group of rank $d$ and derived length $m\ge 2$} is defined as the quotient
	\[
	S_{d,m}\coloneqq F_d/F_d^{(m)},
	\]
	where $F_d^{(1)}=[F_d,F_d]$ and $F_d^{(r)}=\left[F_d^{(r-1)},F_d^{(r-1)}\right]$ for $r=2,\ldots,m$.
	
\end{defn}

Free solvable groups satisfy the following universal property: every finitely generated group with $d$ generators that is solvable of derived length $m$ is a quotient of $S_{d,m}$.  Many structural properties of free solvable groups are closely related to those of wreath products via the classical Magnus embedding theorem \cite{Magnus1939}. This theorem provides an explicit monomorphism
\(
S_{d,m}\hookrightarrow \mathbb{Z}^d\wr S_{d,m-1}\coloneqq \bigoplus_{S_{d,m-1}} \Z^d\rtimes S_{d,m-1}.
\)
Free solvable groups have been studied extensively in the past decades for their algebraic \cite{Bachmut1965,Matthews1966,Baumslag1973,BachmuthMochizuki1982,BachmuthMochizuki1985,Tolstykh2011,Sale2015geometry,Alibabaei2017} algorithmic \cite{RemeslennikovSokolov1970,Baumslag2010,MyasnikovRomankovUshakovVershik2010,Romankov2012,Ushajov2014,Funda2017,Benzezra2019,Benzezra2020} and geometric \cite{Cornulier2011,Vassileva2012,Sale2015,Jacoboni2019,Wang2024} properties.

The study of random walks on free solvable groups goes back to the work of A. Vershik \cite[Section 4]{Vershik2000}, where a geometric interpretation of free solvable groups is introduced (see also \cite{VershikDobrynin2005}) and used to analyze the Poisson boundary of random walks on these groups. The Poisson boundary of free solvable groups has since been studied further: under suitable conditions on the step distribution, descriptions of the Poisson boundary are obtained in \cite[Theorem 2]{Erschler2011}, \cite[Section 6]{LyonsPeres2021}, and \cite[Corollary 1.8]{FrischSilva2024}. In particular, these works provide a description of the space of bounded $\mu$-harmonic functions on $S_{d,m}$ for any non-degenerate probability measure with a finite first moment. Other aspects of random walks on free solvable groups have also been explored. For instance, the asymptotic decay of the return probability to the origin in $S_{d,m}$ is computed in \cite[Theorem 1.1]{SaloffCosteZheng2015}.

The central object of the present paper is the asymptotic entropy of probability measures on free solvable groups. Recall that the \emph{Shannon entropy} of a probability measure $\mu$ on $G$ is defined as \(H(\mu)\coloneqq -\sum_{g\in G}\mu(g)\log(\mu(g)).\)

\begin{defn}\label{defn: asymptotic entropy}
	The \emph{asymptotic entropy} of a probability measure $\mu$ on a countable group $G$ is defined as
	\[
	h(\mu)\coloneqq\lim_{n\to \infty} \frac{H(\mu^{*n})}{n}\in [0,\infty].
	\]
\end{defn}

Asymptotic entropy was introduced by A. Avez \cite{Avez1972}, who showed that a finitely supported probability measure $\mu$ with $h(\mu)=0$ has the Liouville property. Furthermore, for any probability measure with $H(\mu)<\infty$, the \emph{entropy criterion} of Y. Derriennic \cite{Derrienic1980} and V. Kaimanovich and A. Vershik \cite{KaimanovcihVershik1983} states that $h(\mu)=0$ if and only if $\mu$ has the Liouville property.

Our main result establishes the continuity of asymptotic entropy, as a function of the step distribution, on free solvable groups of rank at least $3$.

\begin{thm}\label{thm: free solvable}
	Let $S_{d,m}$ be the free solvable group of rank $d\ge 3$ and derived length $m\ge 2$. Let $\mu$ be a non-degenerate probability measure on $S_{d,m}$ with $H(\mu)<\infty$. Let $\{\mu_k\}_{k\ge 1}$ be a sequence of non-degenerate probability measures on $S_{d,m}$ such that 	$\lim_{k\to \infty}\mu_k(g)=\mu(g)$ for every $g\in S_{d,m}$ and $\lim_{k\to \infty} H(\mu_k)=H(\mu)$. Then, we have $\lim_{k\to \infty} h(\mu_k)=h(\mu)$.
\end{thm}
The assumption $\lim_{k\to\infty} H(\mu_k)=H(\mu)$ is natural in this setting, since asymptotic entropy is known to be upper semicontinuous under pointwise convergence of probability measures (see Proposition \ref{prop: upper semicontinuous entropy}), and the convergence	of Shannon entropies is used to control the entropy of finite convolution powers. 

The question of whether the map $\mu\mapsto h(\mu)$ is continuous was raised in \cite[Section 6, Item 5]{Erschler2004Liouville}. Positive results are known for several classes of groups, including hyperbolic groups \cite[Theorem 2]{ErschlerKaimanovich2013}, \cite[Theorem 2.9]{GouezelSebastienMatheusMacourant2018}, acylindrically hyperbolic groups \cite[Theorem F]{Choi2024}, linear groups and groups acting properly on $\mathrm{CAT}(0)$ cube complexes \cite[Corollary 1.6]{Silva2025}, as well as wreath products of the form $A\wr \Z^d$ with $d\ge 3$ \cite[Theorem 1.2]{Silva2025}. Among these examples, the latter is the only class that includes amenable groups. On the other hand, discontinuity phenomena are known to occur, and all currently known examples arise among amenable groups. The first such examples were constructed by V. Kaimanovich \cite[Theorems 3.1 \& 4.1]{Kaimanovich1983nontrivial}, who exhibited locally finite groups admitting probability measures with positive asymptotic entropy that cannot be approximated by the asymptotic entropy of finitely supported measures. Further examples, including discontinuity for measures supported on a fixed finite generating set, were later obtained by A. Erschler using the wreath product $\Z/2\Z\wr D_{\infty}$ \cite[Lemma 4]{Erschler2011}. We refer the reader to the introduction of \cite{Silva2025} for a more detailed discussion.

The starting point of the proof of Theorem \ref{thm: free solvable} is \cite[Theorem 7.1]{Silva2025} (see also Theorem \ref{thm: continuity asymptotic entropy wreath prods} below), which provides technical conditions under which asymptotic entropy is continuous on a wreath product. The main new difficulty in the present work is extending this argument to iterated wreath products in Theorem \ref{thm: iterated wreath products}, which requires controlling the entropy contribution coming from several layers of the wreath product structure. This is afterwards combined with Magnus' embedding theorem \cite{Magnus1939} (see also Theorem \ref{thm: original formulation Magnus embedding}) to obtain Theorem \ref{thm: free solvable}.

We remark that Theorem \ref{thm: free solvable} does not cover free solvable groups of rank $2$. The reason is that the statement is not expected to hold in full generality for these groups. Indeed, examples of discontinuity of asymptotic entropy can already be found for wreath products of the form $\Z/2\Z\wr \Z^2$. To obtain continuity results in this setting, one needs to restrict the class of probability measures under consideration. In Theorems \ref{thm: wreath prods 12} and \ref{thm: rank2 free solvable}, we prove the continuity of asymptotic entropy on wreath products $A\wr \Z^d$ for $d\in\{1,2\}$, and on free solvable groups of rank $2$, among probability measures supported on a fixed finite subset of the group.

All known examples of discontinuity of asymptotic entropy mentioned so far consist of sequences of probability measures $\{\mu_k\}_{k\ge 1}$ such that $h(\mu_k)=0$ for all $k\ge 1$, while the limit measure $\mu$ satisfies $h(\mu)>0$. However, discontinuity is not restricted to this situation. A simple modification of the above examples, obtained by taking a direct product with a free group, yields examples in which all asymptotic entropies involved are strictly positive. More precisely, in Proposition \ref{prop: discontinuity of asymptotic entropy with positive values} we construct groups $G$ and probability measures $\mu$, $\{\mu_k\}_{k\ge 1}$ on $G$, supported on a fixed finite subset of $G$, such that $h(\mu_k)>0$ for all $k\ge 1$ and $h(\mu)>0$, while $\lim_{k\to\infty}\mu_k(g)=\mu(g)$ for every $g\in G$ but $\lim_{k\to\infty}h(\mu_k)\neq h(\mu)$.

Finally, we remark that certain special cases of Theorem~\ref{thm: free solvable}, under additional assumptions on the measures $\mu$ and $\mu_k$, $k \ge 1$, can be deduced from \cite[Theorem 1.5]{Silva2025}. Indeed, together with the identification of the Poisson boundary obtained in \cite[Corollary 1.8]{FrischSilva2024}, the continuity of asymptotic entropy would follow, among probability measures for which this identification holds, from the weak continuity of the associated harmonic measures on this model of the Poisson boundary; see \cite[Proposition 8.12 and Corollary 8.13]{Silva2025} for the analogous statements in the case of wreath products. However, this argument does not yield Theorem \ref{thm: free solvable} in full generality for arbitrary sequences of probability measures with finite entropy. At this level of generality, there is currently no identification of the Poisson boundary of free solvable groups, and in particular none that simultaneously realizes the Poisson boundaries corresponding to different probability measures within a common Polish model space.

\subsection{Organization}
In Section \ref{section: preliminaries} we recall the main technical results that we use regarding convergence of probability measures. In Section \ref{section: free solvable groups} we prove Theorem \ref{thm: iterated wreath products}, which is the main technical result of this paper, and use it to prove Theorem \ref{thm: free solvable}. Finally, in Section \ref{section: linear quadratic} we discuss continuity and discontinuity results for asymptotic entropy on wreath products $A\wr \Z^d$ with $d\in \{1,2\}$ and on free solvable groups of rank $2$.

\subsection{Acknowledgements}
The author is funded by the Deutsche Forschungsgemeinschaft (DFG, German Research Foundation) under Germany's Excellence Strategy EXC 2044/2 –390685587, Mathematics Münster: Dynamics–Geometry–Structure.
\section{Preliminaries}\label{section: preliminaries}
\subsection{Random walks on groups and entropy}\label{subsection: entropy and Poisson boundary}
Let $\mu$ be a probability measure on a countable group $G$. The \emph{$\mu$-random walk} on $G$ is the Markov chain $\{w_n\}_{n\ge 0}$ with state space $G$, starting at $w_0=e_G$ and with transition probabilities $p(g,h)=\mu(g^{-1}h)$ for $g,h\in G$. Equivalently, for each $n\ge 1$ we have
$$
w_n=g_1g_2\cdots g_n,
$$
where $\{g_i\}_{i\ge 1}$ is a sequence of independent, identically distributed, random variables on $G$ with law $\mu$. The law $\P_{\mu}$ of the chain $\{w_n\}_{n\ge 0}$ is the push-forward of the Bernoulli measure $\mu^{\Z_{+}}$ through the map
\begin{equation*}
	\begin{aligned}
		G^{\mathbb{Z}_{\ge 1}}&\to G^{\Z_{+}}\\
		(g_1,g_2,g_3,\ldots)&\mapsto (w_0,w_1,w_2,w_3,\ldots)\coloneqq (e_G,g_1,g_1g_2,g_1g_2g_3,\ldots).
	\end{aligned}
\end{equation*}
The space $(G^{\Z_{+}},\P_{\mu})$ is called the \emph{space of sample paths} or the \emph{space of trajectories} of the $\mu$-random walk.

In this paper we work with countable partitions of the space of sample paths. The \emph{Shannon entropy} of a countable partition $\rho$ of the space of sample paths $G^{\Z_{+}}$ with respect to the probability measure $\P$ is defined as
\begin{equation*}
	H(\rho)\coloneqq -\sum_{k\ge 1}\P(\rho_k)\log \P(\rho_k).
\end{equation*}

In this paper we will use the following well-known properties of entropy. We refer to \cite[Corollaries 2.5 and 2.6]{MartinEngland1981} and \cite[Section 5]{Rohlin1967} for their proofs.
\begin{lem}\label{lem: basic properties entropy} Consider countable partitions $\rho,\gamma$ and $\delta$ of a Borel space. The following properties hold:
	\begin{enumerate}
		\item \label{item: entropy 1} $	H(\rho\vee \gamma \mid \delta)= H(\rho\mid \gamma\vee \delta)+H(\gamma\mid\delta).$
		\item \label{item: entropy 2} $H(\rho\mid \gamma)\le H(\rho\vee \delta\mid \gamma).$
		\item \label{item: entropy 3} $H(\rho\mid \gamma \vee \delta)\le H(\rho\mid \gamma).$
	\end{enumerate}
\end{lem}

\begin{rem}\label{rem: partitions defined by random variables}
	Throughout this paper we will use the same symbol to denote both a random variable and the partition of the space of sample paths that it defines. More precisely, let $X:(G^{\Z_{+}},\P)\to D$ is a random variable from the space of sample paths with values on a countable set $D$. Then we will identify $X$ with the countable partition $\rho_X$  where two sample paths  $\mathbf{w},\mathbf{w^\prime}\in G^{\mathbb{Z}_{+}}$ belong to the same element of $\rho_X$ if and only if $X(\mathbf{w})=X(\mathbf{w^{\prime}})$.
\end{rem}

\subsection{Escape probability and the asymptotic range}\label{subsection: escape prob and the asymptotic range}

The \emph{range at time $n\ge 1$} of the $\mu$-random walk on the group $G$ corresponds to the random variable $R_n\coloneqq |\{w_0,w_1,\ldots, w_n\}|$ that counts the number of distinct elements of $G$ visited by the trajectory of the random walk up to time $n$. It is a consequence of Kingman's subadditive ergodic theorem \cite[Theorem 5]{Kingman1968} that the limit
$$
R(\mu)\coloneqq \lim_{n\to \infty}\frac{R_n}{n}
$$
exists almost surely and in $L^1(G^{\Z_{+}})$, and that it is a constant. Furthermore, we have $R(\mu)=\P(w_n\neq e_G\text{ for all }n\ge 1)$. This equality is proved in in \cite[Theorem I.4.1]{Spitzer1976} for $G=\Z^d$, $d\ge 1$, and in \cite[Lemma 1]{Dyubina1999} for the general case of a countable group. 

\begin{defn} For each probability measure $\mu$ on a countable group $G$, we define its associated \emph{escape probability} by $\pesc{\mu}\coloneqq \P(w_n\neq e_G \text{ for every }n\ge 1)$.
\end{defn}
The following result guarantees the continuity of the escape probability on groups with at least cubic growth.
\begin{thm}[{\cite[Theorem 1.4]{Silva2025}}]\label{thm: continuity of range} Let $\mu$ be a probability measure on a countable group $G$. Suppose that $\langle \supp{\mu}\rangle_{+}$ contains a finitely generated subgroup of at least cubic growth. Let $\{\mu_k\}_{k\ge 1}$ be a sequence of probability measures on $G$ such that $\lim_{k\to \infty}\mu_{k}(g)=\mu(g)$, for every $g\in G$. Then, we have $\lim_{k\to \infty}\pesc{\mu_k}=\pesc{\mu}$.
\end{thm}

\subsection{Convergence of probability measures}\label{section: convergence of probability measures}
The following result is well-known; a proof can be found in \cite[Lemma 3.1]{Silva2025}.
\begin{lem}\label{lem: pointwise convergence is equivalent to total variation convergence}
	Let $\mu$ be a probability measure on a countable group $G$, and consider a sequence $\{\mu_k\}_{k\ge 1}$ of probability measures on $G$. The following are equivalent.
	\begin{enumerate}
		\item  $\sum_{g\in G}|\mu_k(g)-\mu(g)|\xrightarrow[k\to \infty]{} 0$.
		\item $\mu_{k}(g)\xrightarrow[k\to \infty]{}\mu(g)$, for every $g\in G$.
	\end{enumerate}
\end{lem}
We recall some preliminary results on the convergence of probability measures and Shannon entropies that we will use in the rest of the paper. 

\begin{lem}\label{lem: continuitiy of finite time entropy} Let $\mu$ be a probability measure on a countable group $G$. Consider a sequence of probability measures $\{\mu_k\}_{k\ge 1}$ on $G$ such that $\lim_{k\to \infty}\mu_k(g)=\mu(g)$, for each $g\in G$. Suppose that there is a finite subset $S\subseteq G$ such that $\supp{\mu}\subseteq S$ and $\supp{\mu_k}\subseteq S$ for all $k\ge 1$. Then $\lim_{k\to\infty}H(\mu_k)=H(\mu)$.
\end{lem}

The following result is  proved in \cite[Section 3]{AmirAngelVirag2013}. 	
\begin{lem}\label{lem: convolutions entropy convergence} Let $G$ be a countable group and consider a probability measure $\mu$ on $G$ with $H(\mu)<\infty$. Consider a sequence of probability measures $\{\mu_k\}_{k\ge 1}$ on $G$ such that $\lim_{k\to \infty}\mu_k(g)=\mu(g)$ for all $g\in G$, $H(\mu_k)<\infty$, for all $k\ge 1$, and $\lim_{k\to \infty}H(\mu_k)=H(\mu)$. Then for any $n\ge 1$ it holds that $\lim_{k\to\infty}H(\mu^{*n}_k)=H(\mu^{*n})$.
\end{lem}

Next, we recall the fact that the asymptotic entropy is upper-semicontinuous. The following result is proved in \cite[Proposition 4]{AmirAngelVirag2013}. Previously, this was proved in \cite[Lemma 1]{Erschler2011} with the additional assumption that the supports of all the probability measures in the sequence are contained in a fixed finite set.
\begin{prop}[{\cite[Proposition 4]{AmirAngelVirag2013}}]\label{prop: upper semicontinuous entropy} Let $G$ be a countable group and consider a probability measure $\mu$ on $G$ together with a sequence of probability measures $\{\mu_{k}\}_{k\ge 1}$ on $G$. Suppose that $\lim_{k\to \infty}\mu_{k}(g)=\mu(g)$ for every $g\in G$, and that $\lim_{k\to \infty}H(\mu_k)=H(\mu)$. Then we have $\limsup_{k\to \infty}h(\mu_k)\le h(\mu)$.
\end{prop}

Finally, we prove that convergence of probability measures and convergence of Shannon entropies are preserved when passing to a quotient.

\begin{lem}\label{lem: entropy to projections} Let $G_1$ and $G_2$ be countable groups, and let $\pi:G_1\to G_2$ be an epimorphism. Consider a probability measure $\mu$ and a sequence of probability measures $\{\mu_k\}_{k\ge 1}$ on $G_1$ such that $H(\mu)<\infty$ and $H(\mu_{k})<\infty$, for all $k\ge 1$. Suppose that $\lim_{k\to \infty} \mu_k(g)=\mu(g)$ for all $g\in G_1$, and that $\lim_{k\to \infty} H(\mu_k)=H(\mu)$. Then it also holds that $\lim_{k\to \infty} \pi_{*}\mu_k(g)=\pi_{*}\mu(g)$ for all $g\in G_2$, and that $\lim_{k\to \infty} H(\pi_{*}\mu_k)=H(\pi_{*}\mu)$.
\end{lem}
\begin{proof}
	Thanks to the data processing inequality, we have that $H(\pi_{*}\mu)\le H(\mu)<\infty$ and $H(\pi_{*}\mu_k)\le H(\mu_k)<\infty$ for each $k\ge 1$. Additionally, for every $g\in G_2$ we can use Lemma \ref{lem: pointwise convergence is equivalent to total variation convergence} to obtain
	\[
	|\pi_{*}\mu_k(g)-\pi_{*}\mu(g)|\le \sum_{h\in \pi^{-1}(\{g\})}|\mu_k(h)-\mu(h)|\le \sum_{h\in G_1}|\mu_k(h)-\mu(h)|\xrightarrow[k\to \infty]{}0.
	\]
	
	Let us now show that $H(\pi_{*}\mu_k)\xrightarrow[k\to \infty]{}H(\pi_{*}\mu)$. Thanks to \cite[Lemma 3.2]{AmirAngelVirag2013}, it suffices to show that the sequence $\{\pi_{*}\mu_k\}_{k\ge 1}$ is entropy-tight. Let $\varepsilon>0$ and, using that $\{\mu_k\}_{k\ge 1}$ is entropy-tight, choose a finite subset $F\subseteq G_1$ such that $\sum_{x\notin F}-\mu_k(g)\log(\mu_k(g))<\varepsilon$, for all $k\ge 1$. Let us consider the finite subset $\pi(F)\subseteq G_2$. Then we have
	\begin{align*}
		\sum_{g\notin \pi(F)}-\pi_{*}\mu_k(g)\log\left(\pi_{*}\mu_k(g)\right)&=	\sum_{g\notin \pi(F)}\sum_{h\in  \pi^{-1}(\{g\}) }-\mu_k(h)\log\left(\sum_{x\in  \pi^{-1}(\{g\}) }\mu_k(x)\right)\\
		&\le \sum_{g\notin \pi(F)}\sum_{h\in  \pi^{-1}(\{g\}) }-\mu_k(h)\log\left(\mu_k(h)\right)\\
		&\le \sum_{h\notin F}-\mu_k(h)\log(\mu_k(h))<\varepsilon.
	\end{align*}
	This finishes the proof.
\end{proof}

The next result is the main technical theorem of \cite{Silva2025}, which gives sufficient conditions under which asymptotic entropy is continuous on a wreath product.
\begin{thm}[{\cite[Theorem 7.1]{Silva2025}}]\label{thm: continuity asymptotic entropy wreath prods}
	Let $A$ and $B$ be countable groups and let $\mu$ be a probability measure on $A\wr B\coloneqq \bigoplus_{B}A\rtimes B$ with $H(\mu)<\infty$. Consider a sequence $\{\mu_k\}_{k\ge 1}$ of probability measures on $A\wr B$ with $H(\mu_k)<\infty$ for all $k\ge 1$, and such that 
	\begin{enumerate}[(1)]
		\item\label{item: main thm 1} $\lim_{k\to \infty }\mu_k(g)=\mu(g)$ for each $g \in A\wr B$, and
		\item\label{item: main thm 2} $\lim_{k\to \infty} H(\mu_k)=H(\mu)$.
	\end{enumerate}
	Denote by $\pi:A\wr B\to B$ the canonical epimorphism to the base group $B$. Suppose furthermore that
	\begin{enumerate}[(1)]\setcounter{enumi}{2}
		\item\label{item: main thm 3}  the $\pi_{*}\mu$-random walk on $B$ is transient,
		\item \label{item: main thm 4} $h(\pi_{*}\mu)=0$,
		\item\label{item: main thm 5} $\langle \supp{\pi_{*}\mu}\rangle_{+}$ is symmetric, and
		\item\label{item: main thm 6} $\lim_{k\to \infty}\pesc{\pi_{*}\mu_k}=\pesc{\pi_{*}\mu}$.
	\end{enumerate} Then $\lim_{k\to \infty}h(\mu_k)=h(\mu)$.
\end{thm}
\section{Iterated wreath products and free solvable groups}\label{section: free solvable groups}
The proof of Theorem \ref{thm: free solvable} relies on the fact that free solvable groups embed into iterated wreath products via the Magnus embedding. We therefore begin by establishing a continuity result for asymptotic entropy
on iterated wreath products, which will later be combined with the Magnus embedding to obtain Theorem \ref{thm: free solvable}.

The following result is an extension of Theorem \ref{thm: continuity asymptotic entropy wreath prods} for iterated wreath products and which is proved at the end of this subsection. 

\begin{thm}\label{thm: iterated wreath products}
	Let $A_1,A_2,\ldots, A_{m-1}, B$ be countable groups for some $m\ge 2$. Consider the groups $B_1 \coloneqq B$ and $B_{j+1}\coloneqq A_j\wr B_j$ for $j=1,\ldots,m-1$. Let $\mu$ be a probability measure on $B_m$ with $H(\mu)<\infty$, and consider a sequence of probability measures $\{\mu_k\}_{k\ge 1}$ on $B_m$ with $H(\mu_k)<\infty$ for all $k\ge 1$. Suppose that
	\begin{enumerate}[(1)]
		\item\label{item: iterated main thm 1} $\lim_{k\to \infty }\mu_k(g)=\mu(g)$ for each $g \in B_m$, and
		\item\label{item: iterated main thm 2} $\lim_{k\to \infty} H(\mu_k)=H(\mu)$.
	\end{enumerate}
	For each $j=1,\ldots,m-1$, let us denote by $\pi_{j}:B_m\to B_j$ the canonical epimorphism from $B_m$ to $B_j$. Suppose furthermore that
	\begin{enumerate}[(1)]\setcounter{enumi}{2}
		\item\label{item: iterated main thm 3} the $\pi_{1*}\mu$-random walk on $B$ is transient,
		\item\label{item: iterated main thm 4} $h(\pi_{1*}\mu)=0$, 
		\item\label{item: iterated main thm 5}  for every $j=1,\ldots,m-1$ the semigroup generated by $\supp{\pi_{j*}\mu}$ is symmetric, and
		\item\label{item: iterated main thm 6} for every $j=1,\ldots,m-1$ we have $\lim_{k\to \infty}\pesc{\pi_{j*}\mu_k}=\pesc{\pi_{j*}\mu}$.
	\end{enumerate} Then $\lim_{k\to \infty}h(\mu_k)=h(\mu)$.
\end{thm}

In what follows we will denote a sample path of a random walk on $A\wr B$ as
\[ 
w_n=(\varphi_n,X_n), \text{ where }\varphi_n:B\to A \text{ and }X_n\in B, \text{ for }n\ge 0.
\]
Let us denote the independent increments of the random walk by $g_i=(f_i,Y_i)$, $i\ge 1$, so that
\[
(\varphi_n,X_n)=(\varphi_{n-1},X_{n-1})\cdot (f_n,Y_n).
\]

\subsection{The coarse trajectory}
We begin by introducing the notion of the $t_0$-coarse trajectory of a random walk on a group. For a fixed $t_0\ge 1$. Intuitively, it consists in recording the group element that is hit by the random walk every $t_0$ steps.
\begin{defn}\label{def: coarse trajectory} Let $\mu$ be a probability measure on a group $G$, and let $t_0\ge 1$. Recall that we denote by $\{w_n\}_{n\ge 0}$ a sample path of the random walk on $G$. We define the \emph{$t_0$-coarse trajectory at instant $n$ on the group $G$} as the ordered tuple
	\begin{equation*}
		\mathcal{P}_n^{t_0}(G)=\left(w_{t_0},w_{2t_0},\ldots, w_{\left\lfloor n/t_0\right \rfloor t_0}\right).
	\end{equation*}
\end{defn}

\begin{lem}[{\cite[Lemma 5.2]{Silva2025}}]\label{lem: coarse trajectory has small entropy} Let $\mu$ be a probability measure on a countable group $G$ with $H(\mu)<\infty$ and $h(\mu)=0$. Consider a sequence $\{\mu_k\}_{k\ge 1}$ of probability measures on $G$ with $H(\mu_k)<\infty$ for all $k\ge 1$, such that $\lim_{k\to \infty}\mu_k(g)=\mu(g)$ for each $g\in G$ and that $\lim_{k\to \infty}H(\mu_k)= H(\mu)$. Then for every $\varepsilon>0$, there exist $K,T\ge 1$ such that for all $k\ge K$, $t_0\ge T$ and $n\ge t_0$, we have $H_{\mu_k}(\mathcal{P}_n^{t_0}(G))<\varepsilon n$.
\end{lem}

In the statement of Lemma \ref{lem: starting point for iterated wreath products} below, there are two coarse trajectories that appear. For a probability measure $\mu$ on the wreath product $A\wr B$ and for $N,t_0\ge 1$, we consider $\mathcal{P}_n^{t_0}(B)$ the $t_0$-coarse trajectory of the induced random walk on the base group $B$, and $\mathcal{P}_n^{Nt_0}(A\wr B)$ the $Nt_0$-coarse trajectory of the $\mu$-random walk on $A\wr B$.

\begin{lem}[{\cite[Lemma 7.3]{Silva2025}}]\label{lem: starting point for iterated wreath products} 
	Let  $\mu$ be a probability measure on $A\wr B$ with $H(\mu)<\infty$, and consider a sequence $\{\mu_k\}_{k\ge 1}$ of probability measure on $A\wr B$ with $H(\mu_k)<\infty$ for all $k\ge 1$. Suppose that 
	\begin{enumerate}
		\item $\lim_{k\to \infty}\mu_k(g)=\mu(g)$ for all $g\in A\wr B$ and
		\item $\lim_{k\to \infty}H(\mu_k)=H(\mu)$.
	\end{enumerate}  Denote by $\pi:A\wr B\to B$ the canonical epimorphism, and let us furthermore suppose that
	\begin{enumerate}[(1)]\setcounter{enumi}{2}
		\item $\langle\supp{\pi_{*}\mu}\rangle_{+}$ is symmetric and
		\item $\lim_{k\to \infty}\pesc{\pi_{*}\mu_k}=\pesc{\pi_{*}\mu}$.
	\end{enumerate}
	Then, for every $\varepsilon>0$ there exist $C\ge 0$ and $K,n_0,T\ge 1$ such that for all $k\ge K$, $t_0\ge T$, $N>n_0$ and $n> Nt_0$ we have
	\[
	H_{\mu_k}\Big(\mathcal{P}_n^{Nt_0}(A\wr B)\Big \vert w_n\Big)\le \varepsilon n +(H(\mu)+1)\frac{nn_0}{N}+H_{\mu_k}\Big(\mathcal{P}_n^{t_0}(B)\Big| w_n\Big)+C.
	\]
\end{lem}

The following proposition extends Lemma \ref{lem: starting point for iterated wreath products} to iterated wreath products.
\begin{prop}\label{prop: iterated wreath main prop}
	Let $A_1,A_2,\ldots, A_{m-1}, B$ be countable groups for some $m\ge 2$. Consider the groups $B_1 \coloneqq B$ and $B_{j+1}\coloneqq A_j\wr B_j$, for $j=1,\ldots,m-1$. Let $\mu$ be a probability measure on $B_m$ with $H(\mu)<\infty$, and consider a sequence of probability measures $\{\mu_k\}_{k\ge 1}$ on $B_m$ with $H(\mu_k)<\infty$ for all $k\ge 1$. Suppose that
	\begin{enumerate}[(1)]
		\item $\lim_{k\to \infty }\mu_k(g)=\mu(g)$ for each $g \in B_m$, and
		\item $\lim_{k\to \infty} H(\mu_k)=H(\mu)$.
	\end{enumerate}
	For each $j=1,\ldots,m-1$, let us denote by $\pi_{j}:B_m\to B_j$ the canonical epimorphism from $B_m$ to $B_j$. Suppose furthermore that
	\begin{enumerate}[(1)]\setcounter{enumi}{2}
		\item the $\pi_{1*}\mu$-random walk on $B$ is transient,
		\item $h(\pi_{1*}\mu)=0$, 
		\item   $\langle\supp{\pi_{j*}\mu}\rangle_{+}$ is symmetric for every $j=1,\ldots,m-1$, and
		\item  $\lim_{k\to \infty}\pesc{\pi_{j*}\mu_k}=\pesc{\pi_{j*}\mu}$ for every $j=1,\ldots,m-1$.
	\end{enumerate}	
	Then, for every $\varepsilon>0$ there exist $C\ge 0$, $K,T\ge 1$ and $n_0\ge 1$ such that for all $k\ge K$ and $t_0\ge T$ the following holds.
	Consider any $N_1,N_2,\ldots, N_{m-1}>n_0$ and $n>N_1N_2\cdots N_{m-1}t_0$. Then 
	
	\[
	H_{\mu_k}\Big(\mathcal{P}_n^{N_1N_2\cdots N_{m-1}t_0}(B_m)\Big \vert w_n\Big)\le \varepsilon n +(H(\mu)+1)\sum_{j=1}^{m-1}\frac{n n_0}{N_j}+C.
	\]
\end{prop}
\begin{proof}
	We do the proof by induction. Let us first consider the base case $m=2$. Using Lemma \ref{lem: starting point for iterated wreath products}, we have that for every $\varepsilon>0$, there are $C_1\ge 0$ and $K_1,n_0,T_1\ge 1$ such that for all $k\ge K_1$, $t_0\ge T_1$, $N_1>n_0$ and $n\ge N_1t_0$ we have
	\[
	H_{\mu_k}\left( \mathcal{P}_n^{N_1t_0}(B_2)\Big| w_n \right)\le \frac{\varepsilon}{2}n +(H(\mu)+1)\frac{nn_0}{N}+H_{\mu_k}(\mathcal{P}_n^{t_0}(B)\mid w_n)+C_1.
	\]
	In addition, we are assuming that $h(\pi_{1*}\mu)=0$, so that Lemma \ref{lem: coarse trajectory has small entropy} shows that for each $\varepsilon>0$ there is $K_2,T_2\ge 1$ such that for all $k\ge K_2$, $t_0\ge T$ and $n\ge t_0$ we have $H_{\mu_k}(\mathcal{P}_n^{t_0}(B))<\frac{\varepsilon}{2}n$. Furthermore, note that 
	\(
	H_{\mu_k}\left(\mathcal{P}_n^{t_0}(B)\mid w_n\right)\le 	H_{\mu_k}\left(\mathcal{P}_n^{t_0}(B)\right).
	\)
	
	Then, for each $\varepsilon>0$ we can choose $K\coloneqq\max\{K_1,K_2\}\ge 1$, $T\coloneqq\max\{T_1,T_2\}\ge 1$, and $n_0\ge 1$ such that for all $k\ge K$, $t_0\ge T$, $N_1>n_0$ and $n> N_1t_0$ we have
	\[
	H_{\mu_k}\left( \mathcal{P}_n^{N_1t_0}(B_2)\Big| w_n \right)\le \varepsilon n +(H(\mu)+1)\frac{nn_0}{N_1}+C.
	\]
	This is precisely the statement of the proposition in the case $m=2$.
	
	Let us now consider an arbitrary value $m\ge 3$. Denote $\nu_k\coloneqq \pi_{(m-1)*}\mu_{k}$, $k\ge 1$, and $\nu\coloneqq \pi_{(m-1)*}\mu$. It follows from Lemma \ref{lem: entropy to projections} that the probability measures $\{\nu_k\}_{k\ge 1}$, $\nu$ on $B_{m-1}$ satisfy the hypotheses of the current proposition. Thus, using the inductive hypothesis, we obtain that for every $\varepsilon>0$, there is $C_1\ge 0$, $K_1,T_1\ge T$, and $n_1\ge 1$ such that for all $k\ge K_1$, $t_0\ge T$, any $N_1,\ldots, N_{m-2}>n_1$ and $n>N_1\cdots N_{m-2}t_0$ we have
	\begin{equation}\label{eq: inductive 1}
		H_{\nu_k}\left(\mathcal{P}_n^{N_1N_2\cdots N_{m-2}t_0}(B_{m-1})\mid \pi_{m-1}(w_n)\right)\le \frac{\varepsilon}{2}n+(H(\nu)+1)\sum_{j=1}^{m-2}\frac{nn_1}{N_j}+C_1.
	\end{equation}
	Additionally, from  Lemma \ref{lem: starting point for iterated wreath products} we get that for every $\varepsilon>0$, there are $C_2\ge 0$ and $K\ge K_1$, $n_0\ge n_1$ and $T\ge T_1$ such that for all $k\ge K$, $t_1\ge T$, $N_{m-1}>n_0$ and $n>N_{m-1}t_1$ we have
	
	\begin{equation}\label{eq: inductive 2}
		H_{\mu_{k}}\left(\mathcal{P}_n^{N_{m-1}t_1}(B_{m})\mid w_n\right)\le \frac{\varepsilon}{2}n+(H(\mu)+1)\frac{nn_0}{N_{m-1}}+H_{\mu_k}(\mathcal{P}_n^{t_1}(B_{m-1})\mid w_n)+C_2.
	\end{equation}

	By replacing in Equation \eqref{eq: inductive 2} with $t_1=N_1\cdots N_{m-2}t_0$ and then using Equation \eqref{eq: inductive 1} we obtain that for each $k\ge K$, $t_0 \ge T$, $N_1,N_2,\ldots, N_{m-1}>n_0$ and $n>N_1\cdots N_{m-1}t_0$ we have
	\begin{align*}
		H_{\mu_{k}}\left(\mathcal{P}_n^{N_1\cdots N_{m-1}t_0}(B_{m})\mid w_n\right)&\le \frac{\varepsilon}{2}n+(H(\mu)+1)\frac{nn_0}{N_{m-1}}+\\ &\hspace{15pt }+H_{\mu_k}(\mathcal{P}_n^{N_1\cdots N_{m-2}t_0}(B_{m-1})\mid w_n)+C_2\\
		&\le \frac{\varepsilon}{2}n+(H(\mu)+1)\frac{nn_0}{N_{m-1}}+\\ &\hspace{15pt }+H_{\nu_k}(\mathcal{P}_n^{N_1\cdots N_{m-2}t_0}(B_{m-1})\mid \pi_{m-1}(w_n))+C_2\\
		&\le \varepsilon n+(H(\mu)+1)\frac{nn_0}{N_{m-1}}+\\ &\hspace{15pt }+C_2+(H(\nu)+1)\sum_{j=1}^{m-2}\frac{nn_1}{N_j}+C_1\\
		&\le \varepsilon n+(H(\mu)+1)\frac{nn_0}{N_{m-1}}+\\ &\hspace{15pt }+(H(\mu)+1)\sum_{j=1}^{m-2}\frac{nn_0}{N_j}+C_1+C_2\\
		&= \varepsilon n+(H(\mu)+1)\sum_{j=1}^{m-1}\frac{nn_0}{N_j}+C,
	\end{align*}
	where $C=C_1+C_2$. This finishes the induction, and concludes the proof of the proposition.
\end{proof}

\begin{proof}[The proof of Theorem \ref{thm: iterated wreath products}]
	Let $\varepsilon>0$. By Proposition \ref{prop: iterated wreath main prop} we can choose constants $C\ge0$, $K,T\ge1$ and $n_0\ge1$ such that for all $k\ge K$, $t_0\ge T$, and all $N_1,\ldots,N_{m-1}>n_0$, whenever $n>N_1\cdots N_{m-1}t_0$ we have
	\[
	H_{\mu_k}\!\left(P^{N_1\cdots N_{m-1}t_0}_n(B_m)\,\middle|\, w_n\right)
	\le
	\varepsilon n + (H(\mu)+1)\sum_{j=1}^{m-1}\frac{n n_0}{N_j}+C .
	\]
	
	We now relate the entropy of the coarse trajectory to the entropy of the random walk at a fixed time. Using the definition of the coarse trajectory and the subadditivity of entropy, we obtain
	\begin{align*}
		\left( \frac{n}{N_1\cdots N_{m-1}t_0} -1\right) H_{\mu_k}\left(w_{N_1\cdots N_{m-1}t_0}\right) 	&\le \left\lfloor \frac{n}{N_1\cdots N_{m-1}t_0} \right\rfloor  H_{\mu_k}\left(w_{N_1\cdots N_{m-1}t_0}\right) \\
		&= H_{\mu_k}\left(\mathcal{P}_n^{N_1\cdots N_{m-1}t_0}(B_m)\right)\\
		&\le  H_{\mu_k}\left(\mathcal{P}_n^{N_1\cdots N_{m-1}t_0}(B_m)\mid w_n\right)+H_{\mu_k}(w_n)\\
		&\le  \varepsilon n +(H(\mu)+1)\sum_{j=1}^{m-1}\frac{n n_0}{N_j}+C+H_{\mu_k}(w_n).
	\end{align*}
	
	This implies that 
	\[
	\frac{1}{n}H_{\mu_k}(w_n)\ge 	\left( \frac{1}{N_1\cdots N_{m-1}t_0} -\frac{1}{n}\right) H_{\mu_k}\left(w_{N_1\cdots N_{m-1}t_0}\right) -\varepsilon -(H(\mu)+1)\sum_{j=1}^{m-1}\frac{n_0}{N_j}-\frac{C}{n}.
	\]
	
	By taking $n\to \infty$ and then taking the limit inferior as $k\to \infty$, we obtain
	\[
	\liminf_{k\to \infty} h(\mu_k)\ge  \frac{1}{N_1\cdots N_{m-1}t_0} H_{\mu}\left(w_{N_1\cdots N_{m-1}t_0}\right) -\varepsilon-(H(\mu)+1)\sum_{j=1}^{m-1}\frac{n_0}{N_j}, 
	\]
	where we used Lemma \ref{lem: convolutions entropy convergence} to have the inequality $H_{\mu_k}(w_{N_1\cdots N_{m-1}t_0})\xrightarrow[k\to \infty]{}H_{\mu}(w_{N_1\cdots N_{m-1}t_0})$. Next, we take the limit $N_1\to \infty, N_2\to \infty,\ldots, N_{m-1}\to \infty$ and we get
	\[
	\liminf_{k\to \infty} h(\mu_k)\ge h(\mu) -\varepsilon. 
	\]
	Finally, since this equality holds for all $\varepsilon>0$, we conclude that $\liminf_{k\to \infty} h(\mu_k)\ge h(\mu). $ Together with Proposition \ref{prop: upper semicontinuous entropy}, we conclude that $\lim_{k\to \infty}h(\mu_k)=h(\mu)$. 
	
\end{proof}

\subsection{Free solvable groups}\label{subsection: free solvable}
Let $F_d$ be a free group of rank $d\ge 1$, and consider a normal subgroup $N\lhd F_d$. We now recall the Magnus embedding of $F_d/[N,N]$ into the wreath product $\Z^d\wr F_d/N$. A more detailed exposition on the Magnus embedding can be found in \cite[Subsections 2.2 -- 2.5]{MyasnikovRomankovUshakovVershik2010} and \cite[Section 2]{SaloffCosteZheng2015}.

Let $\Z(F_d/N)$ and $\Z(F_d)$ be the group rings with integer coefficients of $F_d/N$ and $F_d$, respectively. Let us denote by $\pi:\Z(F_d)\to \Z(F_d/N)$ the linear extension of the canonical quotient map $F_d\to F_d/N$. Let $T$ be a free left $\Z(F_d/N)$-module of rank $d$ with basis $\{t_1,\ldots,t_d\}$, and consider the group of matrices
\[
M(F_d/N)\coloneqq \left\{ \begin{pmatrix}
	g & t \\ 0 & 1
\end{pmatrix}  \mid g\in F_d/N \text{ and }t\in T \right\},
\]
where the group operation is matrix multiplication. One can verify with a direct computation that $T$ is isomorphic to the direct sum $\bigoplus_{F_d/N}\Z^d$, and that $M(F_d/N)$ is isomorphic to $\Z^d\wr F_d/N$. The following theorem describes the Magnus embedding of $F_d/[N,N]$ into $M(F_d/N)\cong \Z^d\wr F_d/N$.
\begin{thm}[{\cite{Magnus1939}}]\label{thm: original formulation Magnus embedding} Let $F_d$ be a free group of rank $d\ge 1$, and consider $X=\{x_1,\ldots,x_d\}$ a free basis of $F_d$. Let $N\lhd F_d$ be a normal subgroup and denote by $\pi:F_d\to F_d/N$ the quotient map. Then the homomorphism $\phi:F_d\to M(F_d/N)$ defined by
	\begin{equation*}
		\phi(x_i)\coloneqq \begin{pmatrix}
			\pi(x_i) & t_i \\
			0 & 1
		\end{pmatrix}, \ \ \ \text{ for }i=1,\ldots,d,
	\end{equation*}
	satisfies $\mathrm{ker}(\phi)=[N,N]$. Hence, $\phi$ induces a monomorphism $\widetilde{\phi}:F_d/[N,N]\to M(F_d/N)$.
\end{thm}

\begin{proof}[The proof of Theorem \ref{thm: free solvable}]
	
	Consider any non-degenerate probability measure $\mu$ on $S_{d,m}$ with $H(\mu)<\infty$, and let $\{\mu_k\}_{k\ge 1}$ be a sequence of probability measures with $H(\mu_k)<\infty$, for each $k\ge 1$. Let us suppose that $\lim_{k\to \infty}\mu_k(g)=\mu(g)$ for each $g\in S_{d,m}$ and that $H(\mu_{k})\xrightarrow[k\to \infty]{}H(\mu)$. Recall that in the statement of Theorem \ref{thm: free solvable} we assume that $d\ge 3$.
	
	Using Theorem \ref{thm: original formulation Magnus embedding} we can realize $S_{d,m}$ as a subgroup of $\Z^d\wr \left(\Z^d\wr\left(\cdots \Z^d\wr \left(\Z^d\wr \Z^d\right)\cdots\right)\right)$, the group formed as the iterated wreath product of $m$ copies of $\Z^d$. Since asymptotic entropy and the hypotheses of Theorem  \ref{thm: iterated wreath products} are preserved under the natural projections to the wreath product factors, it suffices to verify that the induced measures satisfy the hypotheses of Theorem \ref{thm: iterated wreath products}.
	
	For each $j=1,\ldots, m-1$ consider the group $B_{j+1}\coloneqq \Z^d\wr B_{j}$, where we define $B_1=\Z^d$. Denote by $\pi_j:S_{d,m}\to B_j$ the canonical projection to the group $B_j$ for each $j=1,\ldots,m-1$. Then we have the following.
	\begin{enumerate}
		\item We are assuming that $\mu$ is non-degenerate, and hence $\supp{\pi_{1*}\mu}$ generates $\Z^d$. Since $d\ge 3$, the $\pi_{1*}\mu$-random walk on $\Z^d$ is transient.
		\item We also have that $h(\pi_{1*}\mu)=0$ as a combination of the fact that random walks on abelian groups have trivial Poisson boundary with the entropy criterion.
		\item For each $j=1,\ldots,m-1$ we have that $\langle \supp{\pi_{j*}\mu}\rangle_{+} =B_j$. Since $d\ge 3$, the group $B_1=\Z^d$ has growth at least cubic, and the rest of the groups $B_j$, for $j=2,\ldots, m-1$, have exponential growth. We conclude from Theorem \ref{thm: continuity of range} that $\lim_{k\to \infty}\pesc{\pi_{j*}\mu_k}=\pesc{\pi_{j*}\mu}$.
	\end{enumerate}
	From this, we can directly apply Theorem \ref{thm: iterated wreath products} and conclude that $\lim_{k\to \infty}h(\mu_k)=h(\mu)$.
\end{proof}
\section{Base groups with linear or quadratic growth}\label{section: linear quadratic}
In this section we turn our attention to the case where the base group of the wreath product has growth either linear or quadratic, which are more delicate. Indeed, we have already mentioned that the escape probability is in general not continuous among finitely supported probability measures on the infinite dihedral group \cite[Lemma 3]{Erschler2011} (see also Example \ref{exmp: D infty and BS 11}), and this leads to analogous counterexamples to continuity of asymptotic entropy on wreath products of the form $A\wr D_{\infty}$ \cite[Lemma 4]{Erschler2011} (see also Example \ref{example: discontinuity on wreath products over Dinfty and BS11}). In contrast, here we prove that one does have continuity among probability measures with a fixed common support when the base group is either $\Z$ or $\Z^2$.
\subsection{Continuity of the escape probability on $\Z$ and $\Z^2$}

The asymptotic entropy on a wreath product is tightly connected with the escape probability on its base group. Therefore, in order to handle base groups of linear or quadratic growth, we first study the continuity of the escape probability on $\mathbb Z$ and $\mathbb Z^2$.

\begin{prop}\label{prop: Zd for strongly finitely supported rw has continuous escape proba}
	Let $d\in \{1,2\}$ and consider a probability measure $\mu$ and a sequence of probability measures $\{\mu_k\}_{k\ge 1}$ on $\Z^d$ such that there exists a finite subset $F\subseteq \Z^d$ with $\supp{\mu}\subseteq F$ and $\supp{\mu_k}\subseteq F$, for all $k\ge 1$. Suppose that $\lim_{k\to \infty}\mu_k(g)=\mu(g)$ for all $g\in \Z^d$. Then $\lim_{k\to \infty}\pesc{\mu_k}=\pesc{\mu}$.
\end{prop}

We will use the following result, which is a consequence of Hoeffding's inequality \cite[Theorem 2]{Hoeffding1963}.

\begin{lem}\label{lem: hoeffding}
	Let $\mu$ be a finitely supported probability measure on $\Z$,  and consider $a,b\in \Z$ such that every $x\in \supp{\mu}$ satisfies $a\le x\le b$. Let us denote $\ell\coloneqq \sum_{x\in \Z}x\mu(x)$, and let $\{w_n\}_{n\ge 0}$ be the $\mu$-random walk on $\Z$. Then for all $n\ge 0$ we have
	
	\begin{equation}\label{eq: Hoeffding}
		\P_{\mu}\left(w_n=0\right)\le 2\exp\left( - \frac{2n\ell^2}{(b-a)^2} \right).
	\end{equation}
\end{lem}
\begin{proof}
	Let $\mu$ be a finitely supported probability measure on $\Z$ as in the statement of the lemma. Then, Hoeffding's inequality \cite[Theorem 2]{Hoeffding1963} (see also \cite[Proposition A.6.1]{Lalley2023}) states that for all $n\ge 0$ and $t>0$ we have
	\[
	\P_{\mu}\left(  \left| w_n-n\ell \right|\ge t \right)\le 2\exp\left( - \frac{2t^2}{n(b-a)^2} \right).
	\]
	
	In particular, evaluating the above in $t=n\ell$ we get
	\begin{equation*}
		\P_{\mu}\left(w_n=0\right)\le\P_{\mu}\left(  \left| w_n-n\ell \right|\ge n\ell \right)\le 2\exp\left( - \frac{2n\ell^2}{(b-a)^2} \right),
	\end{equation*}
	which is the inequality we wanted.
\end{proof}
In the proof of Proposition \ref{prop: Zd for strongly finitely supported rw has continuous escape proba} we will use the following well-known equality that relates the escape probability of the $\mu$-random walk with the expected number of returns to the origin; see, e.g., \cite[Lemma 1.13 (a)]{Woess2000}.
\begin{lem}\label{lem: expected number of visits is 1 over 1-return prob}
	Let $\mu$ be a probability measure on a countable group $G$. Then
	\[	\sum_{n= 0}^{\infty}\mu^{*n}(e_G)=\frac{1}{\pesc{\mu}}.\]
\end{lem}

\begin{proof}[Proof of Proposition \ref{prop: Zd for strongly finitely supported rw has continuous escape proba}]
	Consider $d=1$ or $d=2$. Let $\{\mu_k\}_{k\ge 1}$, $F\subseteq \Z^d$ and $\mu$ be as in the statement of the proposition. For each $k\ge 1$ we consider the mean of $\mu_k$ on $\R^d$, given by $
	\ell_k=\sum_{x\in F}x \mu_k(x)$, as well as $\ell=\sum_{x\in F} x\mu(x)$. Since $F$ is finite, we have the convergence $\ell_k\xrightarrow[k\to \infty]{}\ell$.
	
	If $\ell=0$, then the $\mu$-random walk on $\Z^d$ is recurrent and hence $\pesc{\mu}=0$. In this case, it follows from \cite[Lemma 1 (i)]{Erschler2011} that $\limsup_{k\ge \infty} \pesc{\mu_k}\le \pesc{\mu}=0$, so that $\lim_{k\to \infty} \pesc{\mu_k}=0=\pesc{\mu}$. For the rest of the proof, let us consider the case where $\ell \neq 0.$
	
	Then, it follows from Lemma \ref{lem: hoeffding} together with the continuity of the exponential function that there are $C>0$ and $K\ge 1$ such that for each $k\ge K$ we have 
	\begin{equation}\label{eq: exp upper bound}
		\mu_k^{*n}(0)=\P_{\mu_k}(w_n=0)\le C\exp\left(-Cn\right).
	\end{equation}
	From Equation \eqref{eq: exp upper bound} we have that the sequences $\{\mu_k^{*n}(0)\}_{n\ge 0}$ are bounded above by an integrable sequence, uniformly over all $k\ge K$. From this, we obtain the convergence
	\[
	\sum_{n\ge 0}\mu_k^{*n}(0)\xrightarrow[k\to \infty]{}\sum_{n\ge 0}\mu^{*n}(0),
	\]
	and hence that $\pesc{\mu_k}\xrightarrow[k\to \infty]{}\pesc{\mu}.$
\end{proof}

\begin{rem}
	One should note that Proposition \ref{prop: Zd for strongly finitely supported rw has continuous escape proba} is not true if one does not suppose that all the probability measures in the sequence have a common fixed finite support. For example, consider for each $k\ge 1$ the probability measure $\mu_k$ on $\Z$ defined by
	\[
	\mu_k(1)=\frac{3}{4}\frac{2k}{1+2k}, \mu_k(-1)=\frac{1}{4}\frac{2k}{1+2k}, \text{ and }\mu_k(-k)=\frac{1}{1+2k}.
	\]
	Then the $\mu_k$-random walk on $\Z$ is recurrent, for each $k\ge 1$. Indeed, this follows from the fact that 
	\[
	\sum_{x\in \Z}x\mu_k(x)=\frac{3}{4}\frac{2k}{1+2k}-\frac{1}{4}\frac{2k}{1+2k}-\frac{k}{1+2k}=0.
	\]
	
	Despite the above, the sequence $\{\mu_k\}_{k\ge 1}$ converges pointwise to a probability measure $\mu$ such that the $\mu$-random walk on $\Z$ is transient.
\end{rem}

As we have already mentioned, if one considers groups of linear or quadratic growth other than $\Z$ or $\Z^2$, the analogous generalization of Proposition \ref{prop: Zd for strongly finitely supported rw has continuous escape proba} does not hold.

\begin{exmp}\label{exmp: D infty and BS 11}
	Consider the infinite dihedral group $D_{\infty}=\langle a,b\mid a^2=b^2=1\rangle$. It is shown in \cite[Lemma 4.3]{Gilch2008} that for every adapted probability measure $\mu$ on $D_{\infty}$ with a finite first moment, the $\mu$-random walk on $D_{\infty}$ is recurrent. In particular, this is the case for any probability measure of the form
	\[
	\mu_k=\left(1-\frac{1}{k}\right)\left(p\delta_{ab}+(1-p)\delta_{ba}\right)+ \frac{1}{k}\delta_{a},
	\]
	with $k\ge 1$ and $0\le p\le 1$. Since these measures are recurrent, we have $\pesc{\mu_k}=0$ for all $k\ge 1$. Additionally, we have that $\mu_k\xrightarrow[k\to \infty]{}\mu$ pointwise, where $\mu=p\delta_{ab}+(1-p)\delta_{ba}$ is supported on the subgroup $\langle ab\rangle \cong \Z$. If $p\neq 1/2$ then $\pesc{\mu}>0$, and hence this provides an example for the discontinuity of the escape probability among probability measures supported on a fixed finite subset of $D_{\infty}.$
	
	One obtains a similar example on the Baumslag-Solitar group $\mathrm{BS}(1,-1)=\langle a,b\mid bab^{-1}=a^{-1}\rangle$. This group has quadratic growth and is torsion-free. Here, one can obtain a sequence of probability measures $\mu_k$, $k\ge 1$, supported on a fixed finite subset of $\mathrm{BS}(1,-1)$ that define a recurrent random walk, and which converge pointwise to a probability measure (supported on $\Z^2$) that defines a transient random walk. For example, one can consider
	\[
	\mu_k=\frac{1}{3}\left(1-\frac{1}{k}\right)\left(\delta_{b^2}+\delta_{b^{-2}}+p\delta_{a}+(1-p)\delta_{a^{-1}}\right)+ \frac{1}{2k}\left(\delta_b+\delta_{b^{-1}}\right), \text{for }k\ge 1.
	\]
\end{exmp}

\begin{rem}\label{rem: finite index}
	Let $H$ be a finite-index subgroup of a countable group $G$, and let $\mu$ be a non-degenerate probability measure on $G$. Denote by 
	\[
	\tau \coloneqq \min\{n\ge 1 : w_n \in H\}
	\]
	the first return time to $H$ of the $\mu$-random walk on $G$. Then
	\[
	\mu_H(h) \coloneqq \mathbb{P}_{\mu}(w_{\tau}=h), \qquad h\in H,
	\]
	defines a probability measure on $H$. It holds that the $\mu$-random walk on $G$ is transient if and only if the $\mu_H$-random walk on $H$ is transient. Moreover, if $\mu$ is finitely supported, then $\mu_H$ has an exponential tail (see, e.g., \cite[Theorem 4.9.3]{Yadin2024}). Therefore, by considering the probability measures from Example \ref{exmp: D infty and BS 11}, one can also construct examples of discontinuity of the escape probability on $\mathbb{Z}$ and $\mathbb{Z}^2$. Although these measures are not finitely supported, they still have exponential tails.
\end{rem}

\subsection{Continuity of asymptotic entropy on wreath products over $\Z$ or $\Z^2$}
We now explain how Proposition \ref{prop: Zd for strongly finitely supported rw has continuous escape proba} implies the continuity of asymptotic entropy on wreath products when the base group is $\Z$ or $\Z^2$ (see Example \ref{example: discontinuity on wreath products over Dinfty and BS11}).
\begin{thm}\label{thm: wreath prods 12} Let $A$ be a countable group, $d\in \{1,2\}$, and consider the wreath product $A\wr \Z^d$. Let $F\subseteq A\wr \Z^d$ be a non-empty finite subset. Consider probability measures $\mu$ and $\{\mu_k\}_{k\ge 1}$ on $A\wr \Z^d$ with $\supp{\mu_k}\subseteq F$ for all $k\ge 1$. Suppose that the random walk on $\Z^d$ induced by $\mu$ is transient and that $\lim_{k\to \infty}\mu_k(g)=\mu(g)$ for every $g\in A\wr \Z^d$. Then $\lim_{k\to \infty} h(\mu_k)=h(\mu)$.
\end{thm}

\begin{proof}		
	Thanks to Lemma \ref{lem: continuitiy of finite time entropy} we have $\lim_{k\to \infty} H(\mu_k)=H(\mu)$. Denote by $\pi:A\wr \Z^d\to \Z^d$ the canonical epimorphism. Thanks to Proposition \ref{prop: Zd for strongly finitely supported rw has continuous escape proba} we also have $\lim_{k\to \infty}\pesc{\pi_{*}\mu_k}=\pesc{\pi_{*}\mu}$. Additionally, the Choquet-Deny theorem together with the entropy criterion guarantee that $h(\pi_{*}\mu)=0$. Hence, we can apply Theorem \ref{thm: continuity asymptotic entropy wreath prods} and conclude that $\lim_{k\to \infty}h(\mu_k)=h(\mu)$.
\end{proof}
\begin{rem}
	If one furthermore assumes that $A$ is a hyper-FC-central group, then the statement of Theorem \ref{thm: wreath prods 12} remains true even if the induced random walk on $\Z^d$ is recurrent. Indeed, in such a case the $\mu$-random walk on $A\wr \Z^d$ has a trivial Poisson boundary (see e.g.\ \cite[Proposition 4.9]{LyonsPeres2021}), and hence $h(\mu)=0$. Since asymptotic entropy is upper-semicontinuous (Proposition \ref{prop: upper semicontinuous entropy}) and $h(\mu_k)\ge 0$ for all $k\ge 1$, we conclude that $\lim_{k\to \infty} h(\mu_k)=0=h(\mu)$.
\end{rem}
\begin{thm}\label{thm: rank2 free solvable}
	Consider the free solvable group $S_{2,m}$ of rank $2$ and derived length $m\ge 2$. Let $F\subseteq S_{2,m}$ be a non-empty finite subset. Consider probability measures $\mu$ and $\{\mu_k\}_{k\ge 1}$ on $S_{2,m}$ with $\supp{\mu_k}\subseteq F$ for all $k\ge 1$. Suppose that $\lim_{k\to \infty}\mu_k(g)=\mu(g)$ for every $S_{2,m}$. Then $\lim_{k\to \infty} h(\mu_k)=h(\mu)$.
\end{thm}
\begin{proof}
	Similar to the proof of Theorem \ref{thm: free solvable}, we use Theorem \ref{thm: original formulation Magnus embedding} to realize $S_{2,m}$ as a subgroup of the iterated wreath product of $m$ copies of $\Z^2$. For each $j=1,\ldots,m-1$ define $B_{j+1}\coloneqq \Z^2\wr B_j$, where $B_1\coloneqq \Z^2$, and denote by $\pi_j:S_{2,m}\to B_j$ the canonical projection to $B_j$. 
	Suppose first that the $\pi_{1*}\mu$-random walk on $B_1=\Z^2$ is recurrent. Then the $\pi_{2*}\mu$-random walk on $B_2=\Z^2\wr \Z^2$ has a trivial Poisson boundary and hence $h(\pi_{2*}\mu)=0$ thanks to the entropy criterion. If $m=2$, then the upper-semicontinuity of asymptotic entropy guarantees that $\lim_{k\to \infty}h(\mu_k)=0=h(\mu)$. If $m\ge 3$, then we can re-define $\widetilde{B}_{j+1}\coloneqq \Z^2\wr \widetilde{B}_j$ for $j=1,\ldots, m-2$, where $\widetilde{B}_1\coloneqq \Z^2\wr \Z^2$. Denote by $\widetilde{\pi}_j:S_{2,m}\to \widetilde{B}_j$ the canonical projection to $\widetilde{B}_j$ for every $j=1,\ldots, m-1$. In this case, the group $\widetilde{B}_1$ has exponential growth and hence we have both that the $\widetilde{\pi}_{1*}\mu$-random walk on $\widetilde{B}_1$ is transient, and that $\lim_{k\to \infty}\pesc{\widetilde{\pi}_{1*}\mu_k}=\pesc{\widetilde{\pi}_{1*}\mu}$ thanks to Theorem \ref{thm: continuity of range}. The hypotheses of Theorem \ref{thm: iterated wreath products} are thus verified and we conclude that $\lim_{k\to \infty}h(\mu_k)=h(\mu)$.

	Now suppose that the $\pi_{1*}\mu$-random walk on $B_1=\Z^2$ is transient. Then it follows from Proposition \ref{prop: Zd for strongly finitely supported rw has continuous escape proba} that $\lim_{k\to \infty}\pesc{\widetilde{\pi}_{1*}\mu_k}=\pesc{\widetilde{\pi}_{1*}\mu}$. Again, we verify the hypotheses of Theorem \ref{thm: iterated wreath products} and we conclude that $\lim_{k\to \infty}h(\mu_k)=h(\mu)$.
\end{proof}

\subsection{Examples of discontinuity of asymptotic entropy}\label{subsection: examples of discontinuity of entropy}

We mentioned in the introduction that countable locally finite groups and finitely generated groups of intermediate growth serve as examples where asymptotic entropy is not continuous among infinitely supported probability measures. In this subsection, we explain in more detail examples involving wreath products, where in general asymptotic entropy will be discontinuous even among probability measures supported on a fixed finite subset of the group (Example \ref{example: discontinuity on wreath products over Dinfty and BS11}).

The following class of examples are a generalization of \cite[Lemma 4]{Erschler2011}. 
\begin{exmp}\label{example: discontinuity on wreath products over Dinfty and BS11}
	Using the probability measures mentioned in Example \ref{exmp: D infty and BS 11}, one can construct examples of discontinuity of asymptotic entropy on the groups $\Z/2\Z\wr D_{\infty}$ and $\Z/2\Z\wr \mathrm{BS}(1,-1)$. Indeed, let $G$ be $D_{\infty}$ or $\mathrm{BS}(1,-1)$, and denote by $\{\mu_k\}_{k\ge 0}$ a sequence of probability measures on $G$ that are supported on a fixed finite subset $S\subseteq G$, that define a recurrent random walk, and that converge pointwise to a probability measure $\mu$ on $G$ that defines a transient random walk. Let $A$ be an arbitrary non-trivial hyper-FC-central group, and consider the group $A\wr G$. Consider additionally a finitely supported probability measure $\eta$ on $A$. Then the sequence of probability measures $\nu_k\coloneqq \frac{1}{2}\eta+\frac{1}{2}\mu_k$ is a sequence of probability measures on $G$ that is supported on a fixed finite subset of $G$, and it holds that $h(\nu_k)=0$ for all $k\ge 1$,  $\{\nu_k\}_{k\ge 1}$ converges pointwise to $\nu\coloneqq \frac{1}{2}\eta+\frac{1}{2}\mu$, and $h(\nu)>0$.
	
	Using Remark \ref{rem: finite index}, one can use the above examples to obtain discontinuity of asymptotic entropy among probability measures with an exponential tail on  $\Z/2\Z\wr \Z$ and $\Z/2\Z\wr \Z^2$.
	
\end{exmp}
The above examples are based on finding a sequence of probability measures with zero asymptotic entropy that converges to a probability measure with positive asymptotic entropy. As we show next, taking the direct product of any group from Example \ref{example: discontinuity on wreath products over Dinfty and BS11} with a non-abelian free group yields examples of discontinuity of asymptotic entropy where all involved probability measures have strictly positive asymptotic entropy.

\begin{prop}\label{prop: discontinuity of asymptotic entropy with positive values}
	There is a finitely generated group $G$ and probability measures $\mu,\{\mu_k\}_{k\ge 1}$ on $G$ such that 
	\begin{itemize}
		\item there is a finite subset $F\subseteq G$ such that $\supp{\mu_k}\subseteq F$ for all $k\ge 1$,
		\item  $\mu_k(g)\xrightarrow[k\to \infty]{}\mu(g)$ for all $g\in G$,
		\item $h(\mu_k)>0$ for all $k\ge 1$, and
		\item $\limsup_{k\to \infty}h(\mu_k)<h(\mu)$.
	\end{itemize}
\end{prop}
\begin{proof}
	The idea behind this proof is that the asymptotic entropy of the product random walk decomposes as the sum of the asymptotic entropies on each factor.

	Let $L$ be a finitely generated group that admits probability measures $\mu,\{\mu_k\}_{k\ge 1}$ with $\supp{\mu_k}$ contained in a fixed finite subset of $L$ for all $k\ge 1$, and such that $\{\mu_k\}_{k\ge 1}$ converges pointwise to $\mu$, $\lim_{k\to \infty}H(\mu_k)=H(\mu)$, $h(\mu_k)=0$ for all $k\ge 1$, and $h(\mu)>0$.	For concreteness, one can choose $L$ to be one of the groups considered in Example \ref{example: discontinuity on wreath products over Dinfty and BS11}. Additionally, consider $F_2$ the free group of rank $2$, and let $\eta$ be the probability measure on $F_2$ that is uniform on a symmetric free generating subset of $F_2$.
	
	We consider the group $F_2\times L$, and define probability measures $\nu, \{\nu_k\}_{k\ge 1}$ on $F_2\times L$ by $\nu(x,g)=\eta(x)\mu(g)$ and $\nu_k(x,g)=\eta(x)\mu_k(g)$ for each $k\ge 1$ and $(x,g)\in F_2\times L$. Then, we have
	\begin{itemize}
		\item $\lim_{k\to \infty} \nu_k(x,g)=\eta(x)\mu_k(g)=\nu(x,g)$ for each $(x,g)\in F_2\times L$,
		\item  $H(\nu_k)=H(\eta)+H(\mu_k)<\infty$ for every $k\ge 1$, and
		\item $\lim_{k\to \infty}H(\nu_k)=H(\eta)+H(\mu)=H(\nu)$.
	\end{itemize} 
	
	Additionally, note that for each $n\ge 1$, every $(x,g)\in F_2\times L$ and all $k\ge 1$ we have $\nu^{*n}(x,g)= \eta^{*n}(x)\mu^{*n}(g)$ and $\nu_k^{*n}(x,g)=\eta^{*n}(x)\mu_k^{*n}(g)$. As a consequence, $H(\nu^{*n})=H(\eta^{*n})+H(\mu^{*n})$ and $H(\nu_k^{*n})=H(\eta^{*n})+H(\mu_k^{*n})$ for each $k\ge 1$. From this, we obtain that $\lim_{k\to \infty}h(\nu_k)= \lim_{k\to \infty}h(\eta)+h(\mu_k)=h(\eta)$, whereas $h(\nu)= h(\eta)+h(\mu)>h(\eta)$. Also, we have $h(\nu_k)\ge h(\eta)>0$ for all $k\ge 1$.
\end{proof}
This shows that it is possible to have discontinuity of asymptotic entropy through sequences of probability measures that have positive asymptotic entropy.

	\bibliographystyle{alpha}
	\bibliography{bibliography}
\end{document}